\documentclass[11pt]{amsart}
\usepackage[T1]{fontenc}
\usepackage{lmodern}
\usepackage[a4paper,margin=29mm]{geometry}
\usepackage{amsmath,amssymb,amsthm,mathtools}
\usepackage{microtype}
\usepackage[hidelinks]{hyperref}
\hypersetup{
  pdftitle={Finitistic injective dimension exceeding finitistic projective dimension for a commutative ring},
  pdfauthor={Liang Chen},
  pdfsubject={Big finitistic dimensions of a commutative trivial extension},
  pdfkeywords={finitistic dimension, injective dimension, Boolean ring, trivial extension}
}

\newtheorem{theorem}{Theorem}[section]
\newtheorem{proposition}[theorem]{Proposition}
\newtheorem{lemma}[theorem]{Lemma}
\newtheorem{corollary}[theorem]{Corollary}
\theoremstyle{definition}
\newtheorem{remark}[theorem]{Remark}

\DeclareMathOperator{\FPD}{FPD}
\DeclareMathOperator{\FID}{FID}
\DeclareMathOperator{\pd}{pd}
\DeclareMathOperator{\id}{id}
\DeclareMathOperator{\gldim}{gl.dim}
\DeclareMathOperator{\Hom}{Hom}
\DeclareMathOperator{\Ext}{Ext}
\DeclareMathOperator{\Tor}{Tor}
\DeclareMathOperator{\Ann}{Ann}
\DeclareMathOperator{\supp}{supp}
\DeclareMathOperator{\coker}{coker}
\newcommand{\FF}{\mathbb F_2}
\newcommand{\NN}{\mathbb N}
\newcommand{\cc}{\mathfrak c}
\newcommand{\eps}{\varepsilon}
\newcommand{\into}{\hookrightarrow}
\newcommand{\onto}{\twoheadrightarrow}
\numberwithin{equation}{section}
\allowdisplaybreaks[1]

\title[Finitistic dimensions of a commutative ring]{Finitistic injective dimension exceeding finitistic projective dimension for a commutative ring}
\author{Liang Chen}
\address{School of Mathematical Sciences, Capital Normal University, Beijing, China}
\email{2210501003@cnu.edu.cn}
\date{}

\begin{document}

\begin{abstract}
We construct a commutative ring whose big finitistic projective dimension
is $1$ and whose big finitistic injective dimension is $2$. 
The example contradicts the comparison suggested by Bass for commutative rings.
\end{abstract}

\maketitle

\section{Introduction}

For a ring $A$, its big finitistic projective and injective dimensions are
\[
 \FPD(A)=\sup\{\pd_A X:\pd_A X<\infty\},\qquad
 \FID(A)=\sup\{\id_A X:\id_A X<\infty\},
\]
where $X$ ranges over all $A$-modules. In his study of these invariants,
Bass suggested that
\begin{equation}\label{eq:bass-comparison}
 \FID(A)\leq\FPD(A)
\end{equation}
should hold for every commutative ring $A$ \cite[p.~487]{Bass1960}.
Bass proved this comparison for commutative Noetherian rings
\cite[Corollary~5.5]{Bass1962}. Chase's example, recorded in
\cite[Example~8.5, p.~487]{Bass1960}, is a noncommutative ring $\Lambda$
with $\operatorname{lFPD}(\Lambda)=0<\operatorname{rFID}(\Lambda)$.
In contrast, Bass gives a commutative ring $A$ with
$\FPD(A)>\FID(A)=0$ in \cite[Example~8.6, p.~488]{Bass1960}. We construct a commutative ring
violating \eqref{eq:bass-comparison}.

Here is the construction. Put $k=\FF$ and let
\begin{align*}
 B&=k[x_n:n\geq1]/(x_n^2-x_n:n\geq1),\\
 C&=k[y_n:n\geq1]/(y_n^2-y_n:n\geq1),\qquad S=B\otimes_k C.
\end{align*}
We use the same letters for the variables and their residue classes.
For $\alpha=(\alpha_n)\in\Omega:=\{0,1\}^{\NN}$, set
\begin{equation}\label{eq:construction}
 I_\alpha=(x_n-\alpha_n:n\geq1)S,\qquad
 T_\alpha=S/I_\alpha,\qquad
 E=\bigoplus_{\alpha\in\Omega}T_\alpha,\qquad R=S\ltimes E.
\end{equation}
Thus $R=S\oplus E$ as an abelian group, with multiplication
\begin{equation}\label{eq:ring-multiplication}
 (s,e)(s',e')=(ss',se'+es').
\end{equation}
In particular, $R$ is commutative, $E$ is a square-zero ideal of $R$,
and $R/E\cong S$.

\begin{theorem}\label{thm:main}
The ring $R$ in \eqref{eq:construction} satisfies
\[
 \FPD(R)=1<2=\FID(R).
\]
It has cardinality $\cc=2^{\aleph_0}$, is not coherent, and every
localization of $R$ at a prime ideal is isomorphic to
$k[\eps]/(\eps^2)$. There is an explicitly defined $R$-module $W$ of
cardinality $\cc$ such that $\id_R W=2$.
\end{theorem}

The pair $(1,2)$ is the smallest possible finite pair in a strict
inequality $\FPD<\FID$: Bass also observed that $\FPD(A)=0$ implies
$\FID(A)=0$ for commutative $A$ \cite[p.~487]{Bass1960}.
The noncoherence in Theorem~\ref{thm:main} is established by an explicit
annihilator calculation; no conclusion for the class of coherent
commutative rings is asserted here.

Trivial extensions and their homological dimensions are treated
systematically by Fossum, Griffith, and Reiten \cite{FGR1975}; for
idealization in commutative algebra, see Anderson and Winders
\cite{AndersonWinders2009}. Our projective-dimension calculation is an
elementary proof of a faithfully flat case of the change-of-rings
results in \cite[Section~4.A]{FGR1975}. The upper bound $\FID(R)\leq2$ can also be obtained from Cowley's Corollary~5.4 \cite[p.~370]{Cowley1997},
as explained in Remark~\ref{rem:cowley}. Work on injective modules and finitistic injective
dimension over pullbacks includes \cite{FacchiniVamos1985,Xiong2026}.
The lower bound here comes from a specific module and an explicit
nonzero second Ext class.

The construction uses the two countable sets of variables for different
purposes. Evaluating the $x$-variables indexes the summands of $E$ and
provides the extensions used to control finite injective dimension.
The $y$-variables provide a countable orthogonal family that detects the
nonzero second Ext class. Countable support permits the required
extensions of maps in each $x$-fiber, while the uncountability of
$\Omega$ prevents a simultaneous lifting in the final calculation.

Unless explicitly stated otherwise, all rings are commutative
with identity, and all modules are unital. All finitistic dimensions are big dimensions. We work in
ZFC; no additional set-theoretic hypothesis is used. An $S$-module is
regarded as an $R$-module via $R\onto S$ when its $E$-action is declared
to be zero. Conversely, restriction of an $R$-module to $S$ uses
$s\mapsto(s,0)$. All Ext and Tor groups are taken over the indicated
ring.

\section{The Boolean base and the trivial extension}

The canonical tensor-product isomorphism identifies
\begin{equation}\label{eq:S-presentation}
 S\cong k[x_1,x_2,\ldots,y_1,y_2,\ldots]/
 (x_n^2-x_n,y_n^2-y_n:n\geq1).
\end{equation}
It sends $f(x)\otimes g(y)$ to $f(x)g(y)$; the inverse sends
$x_n$ to $x_n\otimes1$ and $y_n$ to $1\otimes y_n$.
Every polynomial involves finitely many variables. Reduction by the
displayed relations makes it linear in each variable separately.
On any fixed finite set of variables, evaluation at the binary points
identifies the resulting algebra with a finite product of copies of
$k$. In particular, every nonzero element has a nonzero binary
evaluation.

\begin{lemma}\label{lem:hereditary}
The ring $S$ is countable, absolutely flat, coherent, and hereditary.
For every $\alpha\in\Omega$,
\[
 \pd_S T_\alpha=1.
\]
Consequently $\gldim S=\FPD(S)=1$.
\end{lemma}

\begin{proof}
The presentation \eqref{eq:S-presentation} shows that $S$ is countable.
Since its characteristic is $2$ and its generators are idempotent,
every element of $S$ is idempotent. Thus $S$ is a Boolean ring. Every
finitely generated ideal is generated by an idempotent and is a direct
summand of $S$. The usual ideal criterion for flatness shows that all
$S$-modules are flat; the same observation proves coherence.

To establish heredity for arbitrary modules, let $U$ be any ideal of
$S$. Enumerate generators as $U=(a_1,a_2,\ldots)$ and put
\[
 b_n=a_n\prod_{j<n}(1-a_j).
\]
These are pairwise orthogonal idempotents, and
$(a_1,\ldots,a_n)=(b_1,\ldots,b_n)$ for every $n$. Hence
\[
 U=\bigoplus_{n\geq1}Sb_n.
\]
Each summand is projective, so every ideal of $S$ is projective.
The characterization of hereditary rings by projectivity of all
ideals gives $\gldim S\leq1$.

For $I_\alpha$, the preceding construction gives
\begin{equation}\label{eq:orthogonal-x}
 I_\alpha=\bigoplus_{n\geq1}Sf_{\alpha,n},\qquad
 f_{\alpha,n}=(x_n-\alpha_n)
       \prod_{j<n}\bigl(1-(x_j-\alpha_j)\bigr).
\end{equation}
Evaluation at $x=\gamma$ sends $f_{\alpha,n}$ to one precisely when
$\gamma$ first differs from $\alpha$ at position $n$. Therefore every
$f_{\alpha,n}$ is nonzero. The direct sum in
\eqref{eq:orthogonal-x} is not finitely generated, so neither is
$I_\alpha$. If $T_\alpha=S/I_\alpha$ were projective, then $I_\alpha$
would be a direct summand of $S$, and thus a cyclic ideal generated
by an idempotent. This is impossible. The exact sequence
$0\to I_\alpha\to S\to T_\alpha\to0$ proves the assertion.
\end{proof}

For $\alpha\in\Omega$, evaluation in the $x$-variables defines
\begin{equation}\label{eq:rho}
 \rho_\alpha:S\onto C,\qquad s(x,y)\longmapsto s(\alpha,y).
\end{equation}
Its kernel is $I_\alpha$, so $T_\alpha\cong C$ as rings and as
$S$-modules, where $S$ acts on $C$ through $\rho_\alpha$.
We use this identification for coordinate calculations.
Let $\lambda_\alpha:T_\alpha\into E$ be the summand inclusion and
write $u_\alpha=\lambda_\alpha(1)$.

\begin{lemma}\label{lem:local}
The $S$-module $E$ is faithfully flat. Every prime of $R$ has the form
$\mathfrak p_{\alpha,\beta}=\mathfrak m_{\alpha,\beta}\oplus E$, where
\[
 \mathfrak m_{\alpha,\beta}
   =(x_n-\alpha_n,y_n-\beta_n:n\geq1)S,
 \qquad(\alpha,\beta)\in\Omega\times\Omega.
\]
Moreover
\[
 R_{\mathfrak p_{\alpha,\beta}}\cong k[\eps]/(\eps^2).
\]
\end{lemma}

\begin{proof}
Every prime quotient of a Boolean ring is the field $k$. Thus the
maximal ideals displayed above are all the prime ideals of $S$, and
$S_{\mathfrak m_{\alpha,\beta}}=k$. If $\gamma\ne\alpha$, choose $n$
with $\gamma_n\ne\alpha_n$. Then $x_n-\gamma_n$ annihilates
$T_\gamma$ and lies outside $\mathfrak m_{\alpha,\beta}$. Hence
$(T_\gamma)_{\mathfrak m_{\alpha,\beta}}=0$, whereas
$(T_\alpha)_{\mathfrak m_{\alpha,\beta}}=k$.
Localization commutes with direct sums, even with the uncountable
direct sum defining $E$. It follows that
\begin{equation}\label{eq:E-local}
 E_{\mathfrak m_{\alpha,\beta}}\cong k
       \cong S_{\mathfrak m_{\alpha,\beta}}.
\end{equation}
Every $S$-module is flat by Lemma~\ref{lem:hereditary}. If
$E\otimes_S N=0$, equation \eqref{eq:E-local} gives $N_{\mathfrak m}=0$
at every maximal ideal of $S$, and therefore $N=0$. This proves
faithful flatness.

The square-zero ideal $E$ is contained in every prime of $R$.
Primes therefore correspond to those of $S$, and localization of
the trivial extension, together with \eqref{eq:E-local}, gives
$R_{\mathfrak p_{\alpha,\beta}}=k\ltimes k$.
\end{proof}

\begin{proposition}\label{prop:structure}
The ring $R$ has cardinality $\cc$ and is not coherent. Its nilradical
is $E$, its Krull dimension is $0$, and $\pd_R S=\infty$.
\end{proposition}

\begin{proof}
The set $\Omega$ has cardinality $\cc$ and each nonzero module
$T_\alpha$ is countable. Since elements of $E$ have finite support,
$|E|=|R|=\cc$.
The annihilator of the element $(0,u_\alpha)\in R$ is
\[
 \Ann_R(0,u_\alpha)=I_\alpha\oplus E.
\]
Its image under $R\onto S$ is the non-finitely-generated ideal
$I_\alpha$. Thus this annihilator is not finitely generated, and
the annihilator criterion for coherence shows that $R$ is not
coherent.

The ring $S$ is reduced, so $E^2=0$ and $R/E\cong S$ imply that $E$ is
the nilradical. Lemma~\ref{lem:local} gives the Krull dimension.
Finally, localizing $S$ at any prime of $R$ gives the residue field
$k$ over $k[\eps]/(\eps^2)$. The resolution with every differential
equal to multiplication by $\eps$ gives
$\Tor_n^{k[\eps]/(\eps^2)}(k,k)=k$ for all $n\geq0$.
Thus $S$ cannot have finite projective dimension over $R$.
\end{proof}

\section{Finitistic projective dimension}

We record the change-of-rings argument used for the projective
calculation. It is the faithfully flat case of the formulas in
\cite[Lemma~4.4, Theorem~4.9, and Corollary~4.10]{FGR1975}.
We give a module-theoretic proof to specify all the hypotheses.

\begin{proposition}\label{prop:ff-change}
Let $A$ be a commutative ring, let $H$ be a faithfully flat
$A$-module, and put $\Lambda=A\ltimes H$. For every $\Lambda$-module
$X$ of finite projective dimension,
\begin{equation}\label{eq:pd-reduction}
 \pd_\Lambda X=\pd_A(X/HX).
\end{equation}
Furthermore, for every $A$-module $Y$,
\begin{equation}\label{eq:pd-induction}
 \pd_\Lambda(\Lambda\otimes_A Y)=\pd_A Y.
\end{equation}
Consequently $\FPD(\Lambda)=\FPD(A)$.
\end{proposition}

\begin{proof}
Let $P_\bullet\to X$ be a projective $\Lambda$-resolution and put
$H_n=\Tor_n^\Lambda(A,X)$. Since the ideal $H$ annihilates itself,
there is an isomorphism of complexes
\[
 H\otimes_\Lambda P_\bullet
   \cong H\otimes_A(A\otimes_\Lambda P_\bullet).
\]
Flatness of $H$ over $A$ permits tensoring to commute with homology.
The long exact Tor sequence for $0\to H\to\Lambda\to A\to0$
therefore gives
\begin{equation}\label{eq:tor-recurrence}
 H_{n+1}\cong\Tor_n^\Lambda(H,X)\cong H\otimes_A H_n
       \qquad(n\geq1).
\end{equation}
When $\pd_\Lambda X<\infty$, the modules $H_n$ eventually vanish.
Faithful flatness and \eqref{eq:tor-recurrence} imply, successively
downwards, that $H_n=0$ for every $n\geq1$.

Set $Y=X/HX=A\otimes_\Lambda X$. The complex
$A\otimes_\Lambda P_\bullet$ is now a projective $A$-resolution of
$Y$. For any $A$-module $N$, inflated to $\Lambda$, adjunction gives
an isomorphism of Hom complexes and hence
\begin{equation}\label{eq:ext-reduction}
 \Ext_\Lambda^q(X,N)\cong\Ext_A^q(Y,N)
       \qquad(q\geq0).
\end{equation}
In particular $d=\pd_A Y\leq\pd_\Lambda X<\infty$.
For an arbitrary $\Lambda$-module $Z$, both ends of
\[
 0\longrightarrow HZ\longrightarrow Z\longrightarrow Z/HZ
   \longrightarrow0
\]
are annihilated by $H$. Equation \eqref{eq:ext-reduction} and the
associated long exact sequence show that $\Ext_\Lambda^q(X,Z)=0$
for every $q>d$. This proves \eqref{eq:pd-reduction}.

Since $\Lambda=A\oplus H$ is flat over $A$, induction of a
projective $A$-resolution of $Y$ is a projective $\Lambda$-resolution
of $\Lambda\otimes_A Y$. Reducing it modulo $H$ recovers the original
resolution. Thus the positive Tor groups against $A$ vanish, and
the preceding argument gives \eqref{eq:pd-induction} when
$\pd_A Y<\infty$. If the induced module had finite projective
dimension while $\pd_A Y=\infty$, \eqref{eq:pd-reduction} would be
a contradiction. This proves \eqref{eq:pd-induction} in general.
Taking suprema proves the last assertion.
\end{proof}

\begin{corollary}\label{cor:fpd}
For the ring $R$ in \eqref{eq:construction}, $\FPD(R)=1$.
For each $\alpha\in\Omega$, the cyclic module
\begin{equation}\label{eq:Aalpha}
 A_\alpha=R\otimes_S T_\alpha\cong R/RI_\alpha
\end{equation}
has projective dimension $1$ over $R$.
\end{corollary}

\begin{proof}
Apply Proposition~\ref{prop:ff-change} to Lemmas~\ref{lem:hereditary}
and \ref{lem:local}.
\end{proof}

\section{An upper bound for finite injective dimensions}

We first describe a short exact sequence associated to each fiber.
For $\gamma\ne\alpha$, some $x_n-\alpha_n$ acts as one on
$T_\gamma$. Hence
\begin{equation}\label{eq:fiber-ideal}
 I_\alpha E=\bigoplus_{\gamma\ne\alpha}T_\gamma,
 \qquad E/I_\alpha E\cong T_\alpha.
\end{equation}
Since $RI_\alpha=I_\alpha\oplus I_\alpha E$, it follows that
\begin{equation}\label{eq:dual-number-fiber}
 A_\alpha\cong T_\alpha\ltimes T_\alpha
       \cong C[\eps]/(\eps^2).
\end{equation}
Explicitly, the quotient map $q_\alpha:R\onto A_\alpha$ becomes
\[
 q_\alpha(s,e)=\rho_\alpha(s)+e_\alpha\eps,
\]
where $e_\alpha$ denotes the $\alpha$-coordinate of $e$.
There is a short exact sequence of $R$-modules
\begin{equation}\label{eq:self-extension}
 0\longrightarrow T_\alpha\xrightarrow{j_\alpha}A_\alpha
   \xrightarrow{p_\alpha}T_\alpha\longrightarrow0,
 \qquad j_\alpha(c)=c\eps,\quad p_\alpha(c+d\eps)=c.
\end{equation}
Here $E$ acts trivially on the two copies of $T_\alpha$. Crucially,
$j_\alpha$ factors through the projective $R$-module $R$:
\begin{equation}\label{eq:j-factorization}
 T_\alpha\xrightarrow{\lambda_\alpha}E\into R
     \xrightarrow{q_\alpha}A_\alpha.
\end{equation}
Indeed, the composite sends $c$ to $c\eps$.

\begin{lemma}\label{lem:injective-test}
Suppose that an $R$-module $Z$ satisfies
\[
 \Ext_R^q(T_\alpha,Z)=0
       \quad\text{for every $\alpha\in\Omega$ and every $q\geq1$}.
\]
Then $\id_R Z\leq2$.
\end{lemma}

\begin{proof}
Ext takes direct sums in its first variable to products, so
$\Ext_R^q(E,Z)=0$ for $q\geq1$. From $0\to E\to R\to S\to0$,
we obtain
\begin{equation}\label{eq:S-high-ext}
 \Ext_R^q(S,Z)=0\qquad(q\geq2).
\end{equation}
The same holds with $S$ replaced by any projective $S$-module,
viewed as an $R$-module, since it is a direct summand of a direct sum
of copies of $S$. Every $S$-module $N$ has a projective
$S$-resolution of length at most one. Applying the long exact Ext
sequence to that resolution gives
\[
 \Ext_R^q(N,Z)=0\qquad(q\geq3).
\]
For any $R$-module $Y$, the modules $EY$ and $Y/EY$ are annihilated
by $E$ and are therefore inflated $S$-modules. The long exact
sequence associated to $0\to EY\to Y\to Y/EY\to0$ now gives
$\Ext_R^q(Y,Z)=0$ for $q\geq3$. Thus $\id_R Z\leq2$.
\end{proof}

\begin{proposition}\label{prop:fid-upper}
Every $R$-module of finite injective dimension has injective
dimension at most $2$. In particular, $\FID(R)\leq2$.
\end{proposition}

\begin{proof}
Let $\id_R Z<\infty$. The factorization
\eqref{eq:j-factorization} implies that the map
\[
 j_\alpha^*:\Ext_R^q(A_\alpha,Z)\longrightarrow
                      \Ext_R^q(T_\alpha,Z)
\]
is zero for $q\geq1$: it factors through $\Ext_R^q(R,Z)=0$.
The long exact sequence of \eqref{eq:self-extension} consequently
provides injections
\[
 \Ext_R^q(T_\alpha,Z)\into\Ext_R^{q+1}(T_\alpha,Z)
                 \qquad(q\geq1).
\]
Finite injective dimension makes these groups vanish in sufficiently
large degrees. The injections force them to vanish in every positive
degree. Lemma~\ref{lem:injective-test} applies.
\end{proof}

\begin{remark}\label{rem:cowley}
The preceding argument also proves that eventual vanishing of
$\Ext_R^q(S,Z)$, without assuming $\id_R Z<\infty$, implies its
vanishing for every $q\geq2$. Indeed,
\[
 \Ext_R^{q+1}(S,Z)\cong\prod_{\alpha\in\Omega}
                    \Ext_R^q(T_\alpha,Z)\qquad(q\geq1),
\]
and the same connecting injections apply. Thus $1$ is a vanishing
bound for $\Ext_R(S,-)$ in the terminology of
\cite[p.~362]{Cowley1997}. Since $E^2=0$, Cowley's
Corollary~5.4 \cite[p.~370]{Cowley1997} gives the alternative bound
$\FID(R)\leq\FID(S)+1\leq2$. The proof above establishes the bound
directly and also supplies the criterion needed for the witness below.
\end{remark}

\section{A module of injective dimension two}

\subsection{The module and its fiberwise Ext groups}

Using $T_\alpha\cong C$, define $S$-modules
\begin{equation}\label{eq:P-L}
 P=\prod_{\alpha\in\Omega}C,\qquad
 L=\{(c_\alpha)\in P:\supp(c_\alpha)\text{ is at most countable}\}.
\end{equation}
Here the support of a tuple $(c_\alpha)_{\alpha\in\Omega}$ is
$\{\alpha\in\Omega:c_\alpha\ne0\}$.
The action of $s\in S$ on the $\alpha$-coordinate is multiplication
by $\rho_\alpha(s)$. The direct sum $E$ embeds naturally in $L$,
and coordinatewise multiplication defines a map $E\times L\to E$.
Set
\begin{equation}\label{eq:W}
 W=L\oplus L,\qquad
 (s,e)(\ell,p)=(s\ell,sp+e\ell).
\end{equation}
This is an $R$-module: $e\ell$ has finite support, and composition
of two $E$-actions is zero. More explicitly, both sides of
$r(r'w)=(rr')w$ have second coordinate
$ss'p+se'\ell+s'e\ell$, by \eqref{eq:ring-multiplication}.

\begin{lemma}\label{lem:patching}
For every $\alpha\in\Omega$,
\[
 \Ext_S^q(T_\alpha,L)=0\qquad(q\geq1).
\]
Consequently $\Ext_R^q(A_\alpha,W)=0$ for every $q\geq1$.
\end{lemma}

\begin{proof}
Apply $\Hom_S(-,L)$ to
$0\to I_\alpha\to S\to T_\alpha\to0$ and use
\eqref{eq:orthogonal-x}. The group $\Ext_S^1(T_\alpha,L)$ is the
cokernel of
\begin{equation}\label{eq:patching-map}
 L\longrightarrow\prod_{n\geq1}f_{\alpha,n}L,
       \qquad\ell\longmapsto(f_{\alpha,n}\ell)_n.
\end{equation}
For each $n$, the element $f_{\alpha,n}$ acts as the indicator of
\[
 D_{\alpha,n}=\{\gamma\in\Omega:
   \gamma_j=\alpha_j\ (j<n),\ \gamma_n\ne\alpha_n\}.
\]
These sets are pairwise disjoint and partition
$\Omega\setminus\{\alpha\}$. Given $\ell^{(n)}\in f_{\alpha,n}L$,
define $\ell$ by using the coordinates of $\ell^{(n)}$ on
$D_{\alpha,n}$ and by setting its $\alpha$-coordinate equal to zero.
Its support is contained in the countable union
$\bigcup_n\supp\ell^{(n)}$, so $\ell\in L$.
This proves that \eqref{eq:patching-map} is surjective. Higher Ext
groups vanish because $\pd_S T_\alpha=1$.

The module $R=S\oplus E$ is flat over $S$. Inducing a projective
$S$-resolution of $T_\alpha$ therefore gives a projective
$R$-resolution of $A_\alpha$. Adjunction yields
\[
 \Ext_R^q(A_\alpha,W)\cong\Ext_S^q(T_\alpha,W)=0\qquad(q\geq1). \qedhere
\] 
\end{proof}

\begin{lemma}\label{lem:W-upper}
For every $\alpha\in\Omega$ and $q\geq1$,
$\Ext_R^q(T_\alpha,W)=0$. In particular, $\id_R W\leq2$.
\end{lemma}

\begin{proof}
Evaluation at the generator identifies
\[
 \Hom_R(A_\alpha,W)=\Ann_W(I_\alpha).
\]
The annihilator on the right consists exactly of pairs whose two
components are supported at the single coordinate $\alpha$. To see
this, each other coordinate $\gamma$ is acted on as one by some
$x_n-\alpha_n$. Thus this Hom group is identified with $C\oplus C$.
Also $\Ann_W(E)=0\oplus L$, because multiplication by $u_\gamma$
detects the $\gamma$-coordinate of the first component. Consequently
$\Hom_R(T_\alpha,W)$ is the copy of $C$ in the second component
supported at $\alpha$.

Under these identifications, restriction along $j_\alpha$ is
\begin{equation}\label{eq:restriction-surjective}
 \Hom_R(A_\alpha,W)\longrightarrow\Hom_R(T_\alpha,W),
       \qquad(\ell,p)\longmapsto\ell.
\end{equation}
Indeed $j_\alpha(1)=q_\alpha(0,u_\alpha)$, so a map taking $1$ to
$(\ell,p)$ takes $j_\alpha(1)$ to
$u_\alpha(\ell,p)=(0,\ell)$. Single-coordinate elements belong to
$L$, and therefore \eqref{eq:restriction-surjective} is surjective.
The long exact sequence for \eqref{eq:self-extension}, together
with Lemma~\ref{lem:patching}, gives $\Ext_R^1(T_\alpha,W)=0$ and
\[
 \Ext_R^q(T_\alpha,W)\cong\Ext_R^{q+1}(T_\alpha,W)
       \qquad(q\geq1).
\]
Induction proves all the asserted vanishings.
Lemma~\ref{lem:injective-test} gives the injective-dimension bound.
\end{proof}

\subsection{The quotient that detects the lower bound}

We next compute the first Ext group of $S$ with coefficients in $W$.
All identifications in the following calculation are $S$-linear.
Since $E$ is annihilated by itself, a map $E\to W$ has image in
$0\oplus L$. The single-coordinate calculation in the preceding
proof gives
\[
 \Hom_R(E,W)\cong\prod_{\alpha\in\Omega}\Hom_R(T_\alpha,W)
                \cong P.
\]
There is no support restriction on this product: any family
$(c_\alpha)\in P$ defines a map on the direct sum $E$ by
coordinatewise multiplication, and its value on each element of $E$
has finite support. Under $\Hom_R(R,W)=W$, restriction to $E$ is
$(\ell,p)\mapsto\ell$. Its image is $L\subseteq P$. Thus
$0\to E\to R\to S\to0$ gives
\begin{equation}\label{eq:Q}
 \Ext_R^1(S,W)\cong Q:=P/L.
\end{equation}
Lemma~\ref{lem:W-upper} also gives
\begin{equation}\label{eq:SW-vanish}
 \Ext_R^q(S,W)=0\qquad(q\geq2).
\end{equation}

Define pairwise orthogonal nonzero idempotents in $C\subseteq S$ by
\begin{equation}\label{eq:y-idempotents}
 e_n=(1-y_n)\prod_{j<n}y_j\qquad(n\geq1),
 \qquad J=\bigoplus_{n\geq1}Se_n.
\end{equation}
The same orthogonalization as in Lemma~\ref{lem:hereditary} shows that
$J=(1-y_n:n\geq1)S$. Evaluation at $y_n=1$ for all $n$ identifies
$S/J$ with the ring $B$. In what follows, $B$ carries this quotient
$S$-module structure, inflated to $R$.

\begin{proposition}\label{prop:nonzero-ext}
There is a natural isomorphism
\begin{equation}\label{eq:second-ext-cokernel}
 \Ext_R^2(B,W)\cong
 \coker\left[Q\xrightarrow{\phi}\prod_{n\geq1}e_nQ\right],
 \qquad\phi(q)=(e_nq)_{n\geq1}.
\end{equation}
The map $\phi$ is not surjective. Consequently $\id_R W=2$.
\end{proposition}

\begin{proof}
Apply $\Hom_R(-,W)$ to $0\to J\to S\to B\to0$. By
\eqref{eq:SW-vanish},
\[
 \Ext_R^2(B,W)\cong
 \coker\bigl[\Ext_R^1(S,W)\longrightarrow\Ext_R^1(J,W)\bigr].
\]
For each $n$, let $i_n:Se_n\into S$ be inclusion and let
$\pi_n:S\onto Se_n$ be multiplication by $e_n$. Then
$\pi_ni_n=1$ and $i_n\pi_n$ is multiplication by $e_n$.
Contravariance of Ext gives $i_n^*\pi_n^*=1$ on
$\Ext_R^1(Se_n,W)$, whereas $\pi_n^*i_n^*=e_n$ on
$\Ext_R^1(S,W)$.
In a commutative ring, this last action agrees with scalar
multiplication on Ext. Hence $\pi_n^*$ identifies
$\Ext_R^1(Se_n,W)$ with $e_nQ$, and $i_n^*$ becomes
$q\mapsto e_nq$. Taking the product over $n$ proves
\eqref{eq:second-ext-cokernel}.

Choose an alternating sequence $a_n\in k$, for instance $a_n=0$
for even $n$ and $a_n=1$ for odd $n$. Let $t_n\in e_nQ$ be the
class of the constant coordinate family
$(a_ne_n)_{\alpha\in\Omega}\in P$, and put
$t=(t_n)_{n\geq1}$. Suppose that $t=\phi(z+L)$ for some $z\in P$.
For each $n$, the equality $e_n(z+L)=t_n$ means that
\[
 U_n=\{\alpha\in\Omega:e_nz_\alpha\ne a_ne_n\}
\]
is at most countable. The union $\bigcup_nU_n$ is countable, whereas
$\Omega$ is uncountable. Choose $\alpha$ outside this union. Then
\begin{equation}\label{eq:incompatible-values}
 e_nz_\alpha=a_ne_n\qquad\text{for every }n\geq1.
\end{equation}
The element $z_\alpha\in C$ depends on only finitely many variables,
say $y_1,\ldots,y_N$. If $n>N$, multiplication by $e_n$ sets these
variables equal to one, since $y_je_n=e_n$ for $j<n$. Thus
\[
 e_nz_\alpha=\lambda e_n\qquad(n>N),\qquad
 \lambda=z_\alpha(1,1,\ldots)\in k.
\]
Since $e_n\ne0$, equation \eqref{eq:incompatible-values} forces
$a_n=\lambda$ for every $n>N$, contradicting the choice of $a_n$.
Therefore $\phi$ is not surjective, and $\Ext_R^2(B,W)\ne0$.
Lemma~\ref{lem:W-upper} now gives $\id_R W=2$.
\end{proof}

\begin{proof}[Proof of Theorem~\ref{thm:main}]
Corollary~\ref{cor:fpd} gives $\FPD(R)=1$.
Proposition~\ref{prop:fid-upper} gives $\FID(R)\leq2$, and
Proposition~\ref{prop:nonzero-ext} supplies a module of injective
dimension two. The structural assertions follow from
Lemma~\ref{lem:local} and Proposition~\ref{prop:structure}.
Finally $|L|=\cc$: singleton supports give the lower bound, while
the number of countable supports and the number of countable
coordinate assignments are each at most
$\cc^{\aleph_0}=\cc$. Therefore $|W|=|L\oplus L|=\cc$.
\end{proof}

\begin{remark}\label{rem:noninjective-Q}
The same cokernel computes $\Ext_S^1(B,Q)$, by applying
$\Hom_S(-,Q)$ to $0\to J\to S\to B\to0$. Thus
\[
 \Ext_R^2(B,W)\cong\Ext_S^1(B,P/L)\ne0.
\]
In particular, $P/L$ is not injective over the hereditary ring $S$.
This identifies the precise obstruction responsible for the extra
injective dimension over $R$.
\end{remark}

\begin{remark}\label{rem:larger-witness}
The module $L\oplus P$ with the action
$(s,e)(\ell,p)=(s\ell,sp+e\ell)$ also has injective dimension two.
The patching argument for $P$ has no support restriction, and all
subsequent calculations are unchanged. This larger module has
cardinality $2^{\cc}$, whereas the witness $L\oplus L$ above has
cardinality $\cc$. The full product $P$ still occurs as
$\Hom_R(E,W)$ for the smaller witness, because the source $E$ is a
direct sum.
\end{remark}

\paragraph*{Acknowledgements.} The author used ChatGPT to assist in developing the counterexample presented in this paper.

\end{document}